\documentclass[a4paper,12pt,leqno]{amsart} 
\usepackage{amsmath,amsthm,amssymb}  
\usepackage{mathdots}
\usepackage{graphicx}

\numberwithin{equation}{section}

\theoremstyle{plain}
\newtheorem{theorem}{Theorem}[section]
\newtheorem{proposition}[theorem]{Proposition}

\newtheorem{definition}[theorem]{Definition}  

\theoremstyle{definition}  
\newtheorem{example}[theorem]{Example} 
\newtheorem{remark}[theorem]{Remark} 

\newcommand{\C}{\mathbb C}   
\newcommand{\R}{\mathbb R}
\newcommand{\Z}{\mathbb Z}

\newcommand{\al}{\alpha}
\newcommand{\be}{\beta} 
\newcommand{\ga}{\gamma}
\newcommand{\de}{\delta}
\newcommand{\la}{\lambda}
\newcommand{\si}{\sigma} 

\newcommand{\eps}{\epsilon}
\newcommand{\Om}{\Omega}
\newcommand{\De}{\Delta}
\renewcommand{\th}{\theta}
\newcommand{\om}{\omega}

\DeclareMathOperator{\tr}{tr}
\DeclareMathOperator{\rank}{rank}
\DeclareMathOperator{\diag}{diag}

\DeclareMathOperator{\End}{End}

\DeclareMathOperator{\Ker}{Ker}

\DeclareMathOperator{\Span}{Span}

\DeclareMathOperator{\Ad}{Ad}
\DeclareMathOperator{\ad}{ad}

\newcommand{\SL}{\textrm{SL}}
\renewcommand{\sl}{\frak s\frak l}

\newcommand{\g}{{\frak g}}
\newcommand{\h}{{\frak h}}

\newcommand{\no}{\noindent}
      
\newcommand{\pr}{\prime} 
\newcommand{\prr}{{\prime\prime}} 
\newcommand{\prrr}{{\prime\prime\prime}} 
\newcommand{\st}{\ \vert\ }   
\renewcommand{\ll}{\lq\lq}
\newcommand{\rr}{\rq\rq\ }
\newcommand{\rrr}{\rq\rq}  
\renewcommand{\b}{\partial}

\newcommand{\bp}{\begin{pmatrix}} 
\newcommand{\ep}{\end{pmatrix}} 
\newcommand{\bsp}{\left(\begin{smallmatrix}} 
\newcommand{\esp}{\end{smallmatrix}\right)}

\renewcommand{\i}{ {\scriptscriptstyle\sqrt{-1}}\, }
\newcommand{\ii}{ {\scriptstyle\sqrt{-1}}\, }

\begin{document}     

\title[Polytopes]{Polytopes, supersymmetry, and integrable systems
}  
   
\author{Martin A. Guest
and Nan-Kuo Ho}    

\date{}   

\maketitle 

\begin{abstract} We review some links between Lie-theoretic polytopes and field theories in physics, which were proposed in the 1990's.  A basic ingredient is the Coxeter Plane, whose relation to integrable systems and the Stokes Phenomenon has only recently come to light.  We use this to give a systematic mathematical treatment, which gives further support to the physical proposals.
This article is based on a talk which was scheduled to be given at the 
workshop \ll Representations of Discrete Groups and Geometric 
Topology on Manifolds\rrr, Josai University, 12-13 March 2020.  
\end{abstract}

\section{Introduction}\label{1} 

Many links between modern physics and geometry have been discovered in the process of \ll abstracting\rr physical ideas, i.e.\ by inventing (purely mathematical) models in order to test conjectures and look for potentially helpful mathematical techniques.  Zamolodchikov's investigation of massive perturbations of supersymmetric conformal field theories in the 1980's was an example; the vastly overdetermined nature of the problem suggested the possibility that mathematically consistent models could sometimes be determined by symmetry considerations alone.  

One such case was his analysis of a \ll model theory\rr whose symmetries are governed by the exceptional Lie algebra of type $E_8$, and whose perturbations are governed by the sine-Gordon equation.  In this model there are 8 particles, and the masses of the particles (up to an overall scale factor) were predicted from the structure of the $E_8$ root system.

While the physical value of such a model may be open to debate, the nontriviality of the mathematics involved suggests a phenomenon at the interface of physics and mathematics which calls for deeper understanding. The mathematical value of such physical abstractions is, after countless instances, not in any doubt.  
In this spirit, we revisit the situation of Zamolodchikov's example, in order to add more fuel to the fire (so to speak).  

First, we mention some history, with references to the physics literature. 
The general context for models linking particle data with  Lie theory can be found in the work of several groups.   It rests on the foundational ideas linking conformal field theory and the theory of integrable systems, due to Zamolodchikov \cite{Za89} and his coauthors.  The role of Toda theory (a generalization of sine-Gordon) was investigated by Hollowood-Mansfield \cite{HoMa89} and by Eguchi-Yang \cite{EgYa89}.
Specific relations between roots of Lie algebras and the physics of affine Toda theory were observed by Freeman \cite{Fr91}. In particular this work made use of Kostant's results \cite{Ko59} on three-dimensional subalgebras and exponents of Lie algebras. 
This was taken further by Braden et al.\  \cite{BCDS90} and especially by
Dorey \cite{Do91},\cite{Do92}.  
In this work, a remarkable coincidence was observed between a \ll bootstrap prediction\rr of scattering matrices and the existence of a finite-dimensional Lie algebra. 
An extensive review was given by Corrigan  \cite{Co99}.
This direction was primarily algebraic, with a view towards
quantum aspects. 

A related series of developments, with similar ingredients but somewhat different ouput,  was the work of Fendley et al.\   \cite{FLMW91}, and the work of Lerche-Warner  \cite{LeWa91} on polytopic models.  

The classical equations of motion (such as the Toda equations) started to play a more significant role in the work of Cecotti-Vafa \cite{CeVa91},\cite{CeVa92} in which they introduced their topological-antitopological fusion equations (tt* equations) and made specific conjectures regarding the (physically) expected properties of the solutions of the equations.
Of particular interest to geometers were the solutions conjectured to represent deformations of specific conformal data (such as the quantum cohomology rings of specific K\"ahler manifolds).  

Let us introduce some notation.
A central role in this story is played by the root system of a (complex) simple Lie algebra $\g$.  With respect to a fixed Cartan subalgebra $\h$, we have the set of roots $\De$, which is a subset of the dual space $\h^\ast$, and the Weyl group $W$, which acts as a reflection group (Coxeter group) on $\h$, and on $\h^\ast$. 
The deeper structure of these objects was elucidated by Kostant and by Steinberg in several foundational papers in the 1950's and 1960's.  It is this structure which was exploited in the physics models mentioned above.  Zamolodchikov's $E_8$ model subsequently made headlines when experimental evidence was found (see \cite{BoGa11}), and this prompted Kostant to update some aspects of his work from 50 years earlier, in \cite{Ko10}.  

The mechanism at work in the situation of Freeman and Dorey is this: the Coxeter element $\ga$ acts on the finite set $\De$, and each orbit corresponds to a \ll particle\rr (or field excitation).  The mass of the particle is the length of the projection (of any point of the orbit) on a certain real plane in $\h^\ast$, called the Coxeter Plane.  Other quantities such as scattering matrices are encoded in the root data in a physically consistent way.

In the situation of Lerche-Warner, the mechanism depends on a choice of representation
$\th:\g \to \End(V)$. Weight vectors in $V$ correspond to \ll vacuum states\rrr, and \ll solitonic particles\rr correspond to certain edges of the weight polytope (i.e.\ the polytope in $\h^\ast$ whose vertices are the weights).  

This is the background to the project described in this talk, which concerns the underlying differential equations.  These are of two types. The first (section \ref{2}) is a rather elementary linear o.d.e., a generalization of the Bessel equation, which nevertheless exhibits the crucial link with Lie theory and the Coxeter Plane. The relation with \ll particles\rr is described in section \ref{3}, although this is simply an observation; the o.d.e.\ itself does not appear to have any physical origin. The second  is a nonlinear p.d.e., a version of the two-dimensional affine Toda equations, which is a special case of 
the tt* equations of Cecotti-Vafa. It is the tt* equations, and their \ll integrability\rrr, which truly provide the link with physics. 
In section \ref{4} we explain this, and also how it is related to the elementary linear o.d.e.\ 
of section \ref{2}. 

We have discussed solutions of these \ll tt*-Toda equations\rr in more detail in a separate article \cite{Gu21}, and
in much more detail in earlier work \cite{GuLi14},\cite{GIL1},\cite{GIL2},\cite{GIL3},\cite{GH1},\cite{GH2}.  Here we focus on the Lie-theoretic aspects, and the 
\[
\text{particle $\ \ \longleftrightarrow\ \ $ Coxeter orbit}
\]
correspondence. Our main purpose is to show how the {\em solutions} of
the equations provide more evidence for the original physics proposals. We give some
examples related to geometry to illustrate the relevance of these ideas for mathematicians.

\section{The Coxeter Plane (and complex o.d.e.\ theory)}\label{2}

\subsection{The Coxeter Plane}\label{2.1}\ 

Let $\g$ be a complex simple Lie algebra, with corresponding simply-connected Lie group $G$.
Let $\al_1,\dots,\al_l\in\h^\ast$ be a choice of simple roots of $\g$ with respect to the Cartan subalgebra $\h$.  (Thus $\rank\g=\dim\h=l$.) The Weyl group $W$ is the finite group generated by the reflections $r_\al$ in all root planes $\ker\al$, $\al\in\De$.  The Coxeter element is the element $\ga=r_{\al_1}\dots r_{\al_l}$ of $W$. Its conjugacy class is independent of the choice and ordering of the simple roots.  Its order is called the Coxeter number of $\g$, and we denote it by $s$.  
Amongst many other results, the following was proved by Kostant in \cite{Ko59}:

\begin{theorem} The Coxeter element $\ga$ acts on the set of roots $\De$ with $l$ orbits, each containing $s$ elements.
\end{theorem}

The Coxeter Plane is a graphical depiction of a Lie algebra, alongside the (better known) depictions such as the Dynkin diagram and the Stiefel diagram, but it is
harder to define (and harder to find in the literature). 
It is attributed to Coxeter as a visualization of the polytope spanned by $\De$.  

The Coxeter Plane is the result of projecting $\De$  orthogonally onto a certain real plane in $\h^\ast$.
Thus it consists of a finite number of dots in the plane. Conventionally the rays through these dots (starting from the origin) are also drawn.  To define the plane precisely we need some more notation (for more details of what follows, we refer to Appendix B of \cite{GH2}). 

Let $\g=\h\oplus (\oplus_{\al\in\De} \g_\al)$ be the root space decomposition of $\g$ with respect to $\h$. Let $B$ be the Killing form (or a positive scalar multiple). Dual to each $\al\in\De$, we define $H_\al\in\g$ by $B(\ ,H_\al)=\al(\ )$. Then $H_{\al_1},\dots,H_{\al_l}$ is a basis of $\h$,
and it is possible to choose root vectors $e_\al\in\g_\al$ such that $B(e_\al,e_{-\al})=1$
and $[e_\al,e_{-\al}]=H_\al$. 

Let $\psi$ be the highest root. Then $\psi=\sum_{i=1}^l q_i\al_i$ for some positive integers $q_i$, and we have $s=\sum_{i=0}^l q_i$ (where $q_0=1$).  It will be convenient later to write $\al_0=-\psi$. 

Next we define the \ll real\rr subspace of $\h$ to be
\[
\h_\sharp= \{ X\in\h \st \al(X)\in\R \text{ for all } \al\in\De \} = \oplus_{i=1}^l \R H_{\al_i}.
\]
The restriction of $B$ to $\h_\sharp$ is positive definite.  The Coxeter Plane 
will be a certain two-dimensional real subspace of $\h_\sharp$ (or the dual $\h_\sharp^\ast$; we
shall use $\h_\sharp$ from now on), then the roots are projected orthogonally to this plane.  Orthogonal
projection is defined using $B$.

Traditionally, this subspace has been described in rather implicit ways. 
In \cite{Ko10}, Kostant gave the following description.  Let
\[
E_+=\sum_{i=0}^l \sqrt{q_i} e_{\al_i},\quad
E_-=\sum_{i=0}^l \sqrt{q_i} e_{-\al_i}.
\]
We have $[E_+,E_-]=0$.
Then
\[
\h^\pr = \Ker \ad E_+ = \Ker \ad E_-
\]
is another Cartan subalgebra. In Kostant's terminology (from his much earlier paper \cite{Ko59}), 
the Cartan subalgebras $\h,\h^\pr$ are said to be {\em in apposition}.  We denote the set of roots with respect to 
$\h^\pr$ by $\De^\pr$, the root space decomposition by
$\g=\h^\pr\oplus (\oplus_{\be\in\De^\pr} \g_\be)$,
and the real subspace of $\h^\pr$ by $\h^\pr_\sharp$.
(There is no natural identification $\h\cong\h^\pr$, so the roots $\al\in\De$
and the apposition roots $\be\in\De^\pr$ are, a priori, unrelated.)

Kostant's description of the Coxeter Plane in $\h^\pr_\sharp$ is this:  as the complex line $\C E_+$ in $\h^\pr=\h^\pr_\sharp\otimes\C$ is
isotropic with respect to $B$,  it corresponds to an oriented
real $2$-plane $Y$ in $\h_\sharp^\pr$. 

Although this can be used to give
an explicit description of $Y$,  a simpler description was given in Appendix B of our paper \cite{GH2}:

\begin{definition}\label{CP} The Coxeter Plane is the (finite) set in $\C$ consisting of the points
$\{\be(E_+) \st \be \in \De^\pr\}$, together with the rays from the origin which pass through these points.  Each point is labelled with the corresponding set of roots.
\end{definition}

\no Thus, in this description, the underlying $2$-plane is just $\C$.  
It should be noted that, while $E_+$ is (by definition) in $\h^\pr$, it is not
necessarily in $\h^\pr_\sharp$, so the points $\be(E_+)$ are indeed {\em complex} numbers.
It is known that there are precisely $2s$ rays, adjoining rays being separated by $\pi/s$.

The roots situated on any $s$ consecutive rays give a choice of positive roots in $\De^\pr$. The corresponding simple roots are situated on the two extremal rays.
These facts are illustrated in Figure \ref{cox4}, in the case $\g=\sl_4\C$, which we explain next.
\begin{figure}[h]
\begin{center}
\vspace{-2cm}
\includegraphics[scale=0.4, trim= 0 200 0 50]{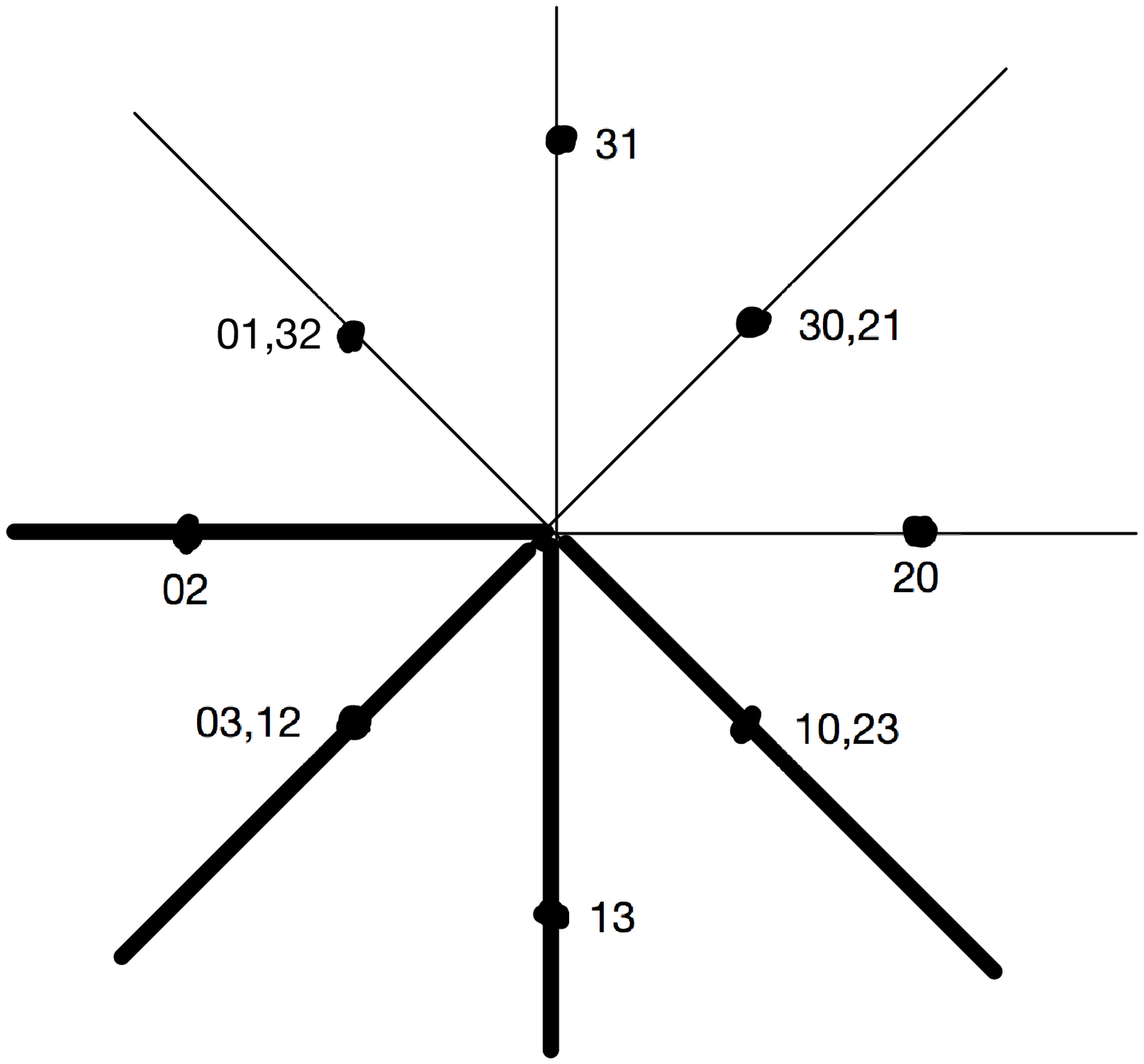}
\end{center}
\caption{Coxeter Plane for $\g=\sl_4\C$.}\label{cox4}
\end{figure}
\begin{example}\label{sl4c}
Let $\h=\{ \diag(x_0,\dots,x_n) \st \sum_{i=0}^n x_i=0 \}$ be
the standard Cartan subalgebra of $\sl_{n+1}\C$, 
and $B(X,Y)=\tr XY$.  
The roots are $x_i-x_j$, for $0\le i\ne j\le n$. 
We choose $\al_1=x_0-x_1,\dots,\al_n=x_{n-1}-x_n$ as
simple roots. Here $l=n$ and $s=n+1$, and all $q_i=1$.  We have
$\al_0=x_n-x_0$ .  
As root vector for the root $x_i-x_j$ we take
$e_{x_i-x_j}=E_{i,j}$, the matrix with $1$ in the $(i,j)$ entry and $0$ elsewhere.

The Coxeter Plane for $\sl_4\C$ 
is shown in Figure \ref{cox4}.  
The apposition Cartan subalgebra and roots, and
the computation of $\be(E_+)$, will be described in section \ref{2.4}.  For
the moment we just state that the apposition
roots $\be\in\De^\pr$ are denoted by symbols $ij$, for $0\le i\ne j\le n$.
In Figure \ref{cox4} we use the same notation for their projections to
the Coxeter Plane.  There are $12$ roots altogether, which project to $8$ points of the plane.

The action of the Coxeter element is given by the cyclic permutation
$(3210)$. This corresponds to rotation through $-\pi/2$.
In Figure \ref{cox4}, as an example, we have chosen $-\pi/4, -\pi/2, -3\pi/4, -\pi$ as the
positive rays (the heavy lines); the corresponding simple roots are
$10,23$ (on the $-\pi/4$ ray) together with $02$ (on the $-\pi$ ray). There are
$8$ such choices; note that this is less than the number of (arbitrary) choices of
positive roots, which is $4!=24$. 
\qed
\end{example}

In \cite{Ko59}, Kostant observed that the adjoint action on $\g$ of a certain element $P_0\in G$ 
preserves $\h^\pr$ and acts there as a Coxeter element. The definition of $P_0$
is:
\[
\textstyle
P_0=\exp(2\pi\i x_0/s),\quad x_0=\sum_{i=1}^l \eps_i
\]
where the $\eps_i\in\h$ are defined by $\al_i(\eps_j)=\de_{ij}$.
Kostant used this to investigate the relation between the Coxeter element and
the exponents $n_1,\dots,n_l$ of $\g$.   
He obtained the following description of the action of the Coxeter element (see section 6 of \cite{Ko85}):

\begin{theorem}   
The set $\{ \al\in\De_+ \st \ga^{-1}\cdot\al \in \De_-\}$ has $l$ elements,
and is a \ll fundamental domain\rr for the action of the Coxeter element $\ga$
on $\De$.  This set can be expressed as 
$
\Pi_2\cup \ga\cdot(-\Pi_1),
$
where  $\Pi=\Pi_1\cup \Pi_2$ is a \ll black and white\rr decomposition of the set of simple roots.
\end{theorem}

\no The meaning of \ll black and white\rr is that one simple root in the Dynkin diagram is declared black, the adjacent simple roots are declared white, and so on.  

Our version of the Coxeter Plane (Definition \ref{CP}) provides a geometrical visualization of this theorem:
$\Pi_2$ and $\Pi_1$ correspond (respectively)  to the roots on the first and last positive rays, and 
the fundamental domain
$\Pi_2\cup \ga\cdot(-\Pi_1)$ consists of the first and second positive rays;
the action of the Coxeter element $\ga$ is given by rotation through $-\pi/2$.
For example, in Figure \ref{cox4}, we have $\Pi_2=\{10,23\}$ (the first ray) and $\Pi_1=\{02\}$ (the last ray).
Then $\Pi_2\cup \ga\cdot(-\Pi_1)= \{ 10,23,13\}$. 

A more significant aspect of Definition \ref{CP} is its source: a certain ordinary differential equation, which will lead us to the relation with physics.

\subsection{The Stokes Phenomenon}\label{2.2}\ 

We introduce  a certain ordinary differential equation in the complex variable $\la$.  The o.d.e.\ is linear, but its coefficients have singularities at $\la=0$ and $\la=\infty$.  A concrete example will be presented shortly, but in general Lie-theoretic terms it can be written
\[
\hat\om = F^{-1} F_\la \, d\la,
\]
where $F=F(\la)$ is a $G$-valued function (to be found), and where $\hat\om=\hat\om(\la)$ is
the (meromorphic) $\g$-valued $1$-form
\begin{equation}\label{omegahat}
\hat\om = 
\left[
-\tfrac{1}{\la^2}
\tfrac{sz}{N} \, \eta
 + 
 \tfrac{1}{\la} m
\right]
d\la,
\end{equation}
in which the coefficients of $\la^{-2},\la^{-1}$ are defined as follows.

First, $\eta=\sum_{i=0}^l c_i z^{k_i} e_{-\al_i} \in\g$, where the $c_i>0$ and the $k_i\in\R$ are
constants, and $z\in\C^\ast$ is a parameter.  Then $N=s+\sum_{i=0}^l q_i  k_i$ where $s$ is
the Coxeter number and the $q_i$ are as defined earlier.  Finally $m\in\h_\sharp$ is defined by the
conditions $\al_i(m)= \tfrac{s}{N} (k_i+1) - 1$ for $1\le i\le l$.  

For fixed $c_i,k_i,z$, this $\hat\om$ is a Lie algebra-valued $1$-form on $\C^\ast$. 
We regard it as a connection form, i.e.\ we consider the connection $\nabla=d+\hat\om$ in the trivial bundle over $\C^\ast$ with fibre $\g$. It is meromorphic with respect to the coordinate $\la$ of $\C^\ast$, with poles of order $2,1$ at $\la=0,\infty$. 

The $c_i,k_i,z$ do not play any role at this point, so let us (temporarily) set
$c_i=1,k_i=0,z=1$ and consider the case $\g=\sl_4\C$ in order to explain in very concrete
terms the relation with the Coxeter Plane in Figure \ref{cox4}.
Thus we have 
\[
\hat\om =
\left[
-\tfrac{1}{\la^2}
\bp
0 & 0 & 0 & 1\\
1 & 0 & 0 & 0\\
0 & 1 & 0 & 0\\
0 & 0 & 1 & 0
\ep
+
\tfrac{1}{\la}
\bp
\!m_0\!\! & 0 & 0 & 0\\
0 & \!\!m_1\!\! & 0 & 0\\
0 & 0 & \!\!m_2\!\! & 0\\
0 & 0 & 0 & \!\!m_3\!
\ep
\right]
d\la.
\]
The values of the $m_i$ will not concern us for the moment; let us take them
to be arbitrary real numbers.

From the definition $\hat\om=F^{-1} F_\la d\la$, it follows that the transpose $F^T$ is a fundamental solution matrix for the o.d.e.\ system
\begin{equation}\label{4x4ode}
Y_\la=
\left[
-\tfrac{1}{\la^2}
\bp
0 & 1 & 0 & 0\\
0 & 0 & 1 & 0\\
0 & 0 & 0 & 1\\
1 & 0 & 0 & 0
\ep
+
\tfrac{1}{\la}
\bp
\!m_0\!\! & 0 & 0 & 0\\
0 & \!\!m_1\!\! & 0 & 0\\
0 & 0 & \!\!m_2\!\! & 0\\
0 & 0 & 0 & \!\!m_3\!
\ep
\right]
Y,
\end{equation}
where
$Y=(y_0,y_1,y_2,y_3)^T$.
By elementary o.d.e.\ theory, such an $F$ exists locally near any chosen $\la_0$, and is unique up to multiplication on the left
by an element of $\SL_4\C$.  The o.d.e.\ is equivalent to a
scalar o.d.e.\ of order $4$, of Bessel type.  

What can be said about the solution $F$? If $\la_0\ne0,\infty$ then $F$ has
a Taylor series expansion at $\la_0$, which is always convergent in a nonempty open neighbourhood of $\la_0$, and
the normalization $F(\la_0)=I$ gives a \ll canonical\rr solution around the point $\la_0$.

If $\la_0=\infty$ (the pole of order $1$) then the Frobenius Method produces a (locally) convergent series expansion,
possibly involving $\log\la$; such solutions are holomorphic on sectors of width $<2\pi$ at
$\la_0$. In this case various normalizations are possible, depending on the values of the $m_i$, but 
this is still in the realm of elementary o.d.e.\ theory.  

If $\la_0=0$ (the pole of order $2$) then a generalization of the Frobenius Method produces a series expansion,
but this series is almost always divergent; it is just a formal solution. Nevertheless,
it is a classical fact
that on open sectors of certain width ($5\pi/4$ in this\footnote{
These sectors are also required to begin at angles of the form $n\pi/4$, $n\in\Z$.
}
example)  there is a unique holomorphic
solution whose asymptotic expansion is the given formal solution.  These
sectors are called Stokes sectors, and the fact that the
solution depends on the sector 
is called the Stokes Phenomenon.  
These solutions are just as canonical as the
normalized solution at $\la_0=\infty$.

In all cases the local solutions are analytic on the sectors where they were originally specified, and can be continued to the universal covering space $\tilde \C^\ast \ (\cong \C)$  of $\C^\ast$.  In this sense, all solutions are \ll essentially the same\rrr, as they differ only by (multiplication by) constant matrices. The space of solutions on $\tilde \C^\ast$ to this (linear!) o.d.e.\ is a four-dimensional vector space. 

However the various \ll canonical\rr solutions that we have mentioned play quite different roles.  In particular the solutions at zero and infinity (by construction) indicate explicitly their asymptotic behaviour at those singular points. This information is not visible if the solution is obtained merely by analytic continuation of a Taylor series solution at some $\la_0\ne0,\infty$.
Thus, when analyzing an o.d.e., the fundamental problem (in general, nontrivial) is to determine the constant matrices which
\ll connect\rr the various canonical solutions at the poles.  

At $\la_0=0$ this means the determination of the \ll Stokes matrices\rr (or \ll Stokes factors\rrr) which relate
solutions on different Stokes sectors. We also have connection matrices which relate such solutions to the canonical solution at $\la_0=\infty$.  
These matrices are known collectively as the \ll monodromy data\rr of (\ref{4x4ode}). 
The benefit of the monodromy data (when properly formulated) is that it is intrinsic; it can be used to parametrize the \ll moduli space\rr of such equations, i.e.\ the space of equivalence classes of equations under gauge transformations.  

\subsection{The Coxeter Plane and the Stokes Phenomenon}\label{2.3}\ 

{\em The Coxeter Plane turns out to be a diagram of the Stokes sectors for our differential equation.}  We shall explain this first for equation (\ref{4x4ode}), then in the general case.  

The Stokes sectors for (\ref{4x4ode}) depend on the eigenvalues of 
\[
\bp
0 & 1 & 0 & 0\\
0 & 0 & 1 & 0\\
0 & 0 & 0 & 1\\
1 & 0 & 0 & 0
\ep
=E_+,
\]
which are the $4$-th roots of unity $1,\om,\om^2,\om^3$. According to the classical theory (\cite{FIKN06}; see section 4.2 of \cite{GIL1} for a similar example),
these sectors are of the form
\[
S_{\th^\pr,\th^\prr}=(\th^\pr -\tfrac\pi2, \th^\prr +\tfrac\pi2)
\]
where the angles $\th^\pr,\th^\prr$ are consecutive angles in the set
\[
\{ \arg \om^i - \om^j \st 0\le i\ne j \le 3 \}
=
\{ 0, \tfrac{\pi}4, \tfrac{\pi}2,\tfrac{3\pi}4,\pi,\tfrac{5\pi}4,\tfrac{3\pi}2,\tfrac{7\pi}4 \},
\]
i.e.\ the arguments of all differences of (distinct) eigenvalues.  
As consecutive angles are separated by $\pi/4$, this produces the Stokes sectors of width $5\pi/4$ referred to earlier.  The rays with these angles are the rays of the Coxeter Plane in Figure \ref{cox4}.

We shall see in a moment that the points in the Coxeter Plane giving rise to those rays --- the roots --- determine the {\em shape} of the
Stokes matrices.  On the other hand, the {\em values} of the Stokes matrices cannot be predicted from the Coxeter Plane. They depend on the remaining coefficients of the equation (the $m_i$ in this example).

Recall that we have a canonical solution $F^T$ of the o.d.e.\ on the sector $S_{\th^\pr,\th^\prr}$.  Let us call this $F^T_{\th^\pr,\th^\prr}$.  On the next sector $S_{\th^\prr,\th^\prrr}$ we have $F^T_{\th^\prr,\th^\prrr}$.  The Stokes factor $Q_{\th^\prr}$ is defined by
\[
F^T_{\th^\prr,\th^\prrr} = F^T_{\th^\pr,\th^\prr} \ Q_{\theta^{\pr\pr}}.
\]
As these (like any other) solutions extend analytically to the universal cover, we should regard the angles as being in $\R$, rather than in $[0,2\pi)$.  Thus we have an infinite sequence of Stokes sectors and Stokes factors (in fact any two consecutive Stokes factors determine the rest, in our situation, because of the symmetries of the equation).   The product of any four consecutive Stokes factors
$ Q_{\th_1}  Q_{\th_2}  Q_{\th_3}  Q_{\th_4}$
 is a Stokes matrix.

In the classical approach, the Stokes matrices can be made triangular. For Lie-theoretic o.d.e.\ (as in our situation), there is a root-theoretic description of the shape, 
first observed by Boalch \cite{Bo02}.  This avoids the arbitrary choice of diagonalization of $E_+$; we simply work with the Cartan subalgebra containing $E_+$ and the roots with respect to that.
The Stokes factor $Q_{\th^\pr}$ then lies in the subgroup determined by the roots on the ray $\th^\pr$; 
the Stokes matrices lie in the Borel subgroups determined by four consecutive rays.  

Let us return now to the Lie-theoretic connection form $\hat\om$ for the Lie algebra $\g$.  The Stokes sectors are determined in exactly the same way as in the case $\g=\sl_4\C$, using the arguments of the complex numbers $\be(E_+)$ where $\be\in\De^\pr$. This is where our Definition \ref{CP} comes from.

\begin{remark}\label{lead}  To be precise, it is necessary to use the Cartan subalgebra and roots corresponding to
the coefficient of $\la^{-2}$ in $\hat\om$, i.e.\ $-\frac{sz}N \eta$, or the \ll transpose version\rr defined by $-\frac{sz}N \eta^T= -\frac{sz}N \sum_{i=0}^l c_i z^{k_i} e_{\al_i}$.  However, as in section 6 of \cite{GH2}, this can be conjugated easily to $-E_-$, or in the transpose version to $-E_+$.  As any
of $\pm E_\pm$ give the same set of points, for simplicity we have used $E_+$ in Definition \ref{CP}.
\end{remark}

For $\hat\om$, the Stokes factors can be computed explicitly.  Summarizing:

\begin{theorem}\label{guitli}
At the singular point $\la_0=0$ of the connection form $\hat\om$ we have:

(i) The Stokes sectors are the sectors 
$S_{\th^\pr,\th^\prr}=(\th^\pr -\tfrac\pi2, \th^\prr +\tfrac\pi2)$
where the angles $\th^\pr,\th^\prr$ are consecutive angles in the set
$\{\arg \be(E_+) \st \be\in\De^\pr \}$.  Each sector has width $\tfrac{\pi}{2}+\tfrac{\pi}{s}$. 

(ii) The Stokes factor $Q_{\th^\prr}$ is of the form
$Q_{\th^\prr}=\exp\left(
\sum_{ \be } s_\be e_\be
\right)$
where the sum is over $\be\in\De^\pr$ satisfying $\arg \be(E_+) = \th^\prr$.

(iii) The Stokes coefficient $s_\be$ is given by an explicit polynomial
expression in $\om^{k_0},\dots,\om^{k_l}$ (where $\om=e^{ 2\pi\i/s }$).
\end{theorem}

Proofs and further details can be found in section 6 of \cite{GH2}).

We emphasize that such a simple description of the Stokes Phenomenon is
possibly only because of the highly symmetrical nature of $\hat\om$.  For example, the
formal monodromy is trivial, and this implies that $Q_{\th^\prr}=Q_{2\pi+\th^\prr}$. It is
the symmetries which allow us to conclude that all of the $2s$ Stokes factors are
determined by any two consecutive ones. 
The symmetries imply that the $s_\be$ are all real, that  $s_\be = s_{-\be}$, and
that  $s_{\be_1}=  s_{\be_2}$ if $\be_1,\be_2$ belong to the same orbit of the Coxeter 
element. 

\begin{definition}\label{groups}
Let $[\be]$ denote the the projection to the Coxeter Plane of the
Coxeter orbit of $\be$.   As $s_{\be}$ depends only on
$[\be]$, we may write $s_\be=s_{[\be]}$.  After choosing an ordering 
of the $l$ orbits, we denote the corresponding $s_{[\be]}$ by $s_1,\dots,s_l$. 
\end{definition}

With this notation, the polynomials in (iii) of Theorem \ref{guitli} are given by
$s_i=\chi_i(M)$, where $\chi_1,\dots,\chi_l$ are the
characters of the basic irreducible representations of $G$, and
$M=e^{ 2\pi\i(m+x_0)/s  }$ (see sections 5.2 and 6 of \cite{GH2}). 

To summarize, we have seen in this section how our o.d.e\ exhibits a surprisingly concrete relation with the Lie theory of the Coxeter Plane.  The link with physics, to be described in sections \ref{3} and \ref{4}, is perhaps still more surprising.

\subsection{Appendix: apposition data for $\g=\sl_{n+1}\C$}\label{2.4}\ 

Although the \ll apposition Cartan subalgebra\rr $\h^\pr$ is convenient for theoretical purposes, it makes matrix calculations (with root vectors, for example) much more complicated.  This is not much of a disadvantage as one generally needs to know eigenvalues rather than eigenvectors.  However, to reassure the nervous reader, we give in this section a simple matrix description of the apposition roots and root vectors in the case $\g=\sl_{n+1}\C$.  This will also explain how our diagrams of Coxeter Planes were constructed.

Starting from the standard diagonal Cartan subalgebra $\h$ (as in Example \ref{sl4c}),
the apposition Cartan subalgebra is
\[
\h^\pr =\Span_\C\{ E_+,E_+^2,\dots,E_+^n\}.
\]
It is easy to diagonalize $E_+$ (and hence $\h^\pr$), because the Vandermonde matrix 
\[
\Om=
\bp
1 & 1 & \cdots & 1 & 1
\\
1 & \om & \cdots & \om^{n-1} & \om^n
\\
\vdots & \vdots &   & \vdots & \vdots 
\\
1 & \om^{n} & \cdots & \om^{(n-1)n} & \om^{n^2}
\ep,
\quad
\om=e^{2\pi\i/(n+1)  }
\]
has the property $\Om^{-1} E_+ \Om = d_{n+1}=\diag(1,\om,\dots,\om^n)$. 
(We have $E_+ \Om = \Om d_{n+1}$, because $E_+$ has eigenvalues $1,\om,\dots,\om^n$,
and the columns of $\Om$ serve as eigenvectors.)

Thus we have the \ll diagonalization map\rr
\[
\Ad \Om^{-1}: \h^\pr \to \h.
\]
Using this identification, the root $x_i-x_j$ with respect to $\h$
corresponds to a root with respect to $\h^\pr$ which we denote by $ij$. By
definition we have
\[
ij(-E_+)= (x_i-x_j)(-d_{n+1}) = \om^j-\om^i.
\]
This is how the Coxeter Plane in Example \ref{sl4c} was obtained.  
For consistency with \cite{GH2}
we are using $ij(-E_+)$ here, rather than 
$ij(E_+)$, as permitted by Remark \ref{lead}.

With respect to the basis $\{ E_+,E_+^2,\dots,E_+^n\}$ of $\h^\pr$, the root $ij$ 
is given by
\[
\textstyle
ij \left(
\sum_{k=1}^n a_k E_+^k
\right)
=
\sum_{k=1}^n a_k (\om^{ki}-\om^{kj}),
\]
as $ij(E_+)=(x_i-x_j)(d_{n+1})$.
It follows that 
\[
\textstyle
\h^\pr_\sharp=
\left\{
\sum_{k=1}^n a_k E_+^k \st a_k=\bar a_{n+1-k}
\right\}.
\]
The \ll classical\rr description of the Coxeter Plane can be computed from this: it is
the real $2$-plane in $\h^\pr_\sharp$ corresponding to the isotropic complex line
$\C E_+$ in $\h^\pr$, namely
\[
\Span_\R\{ E_+ + E_-, \ii(E_+ - E_-) \}
\]
(note that $E_-=E_+^{-1}=E_+^n$). The complexification of this $2$-plane is 
$\Span_\C\{ E_+,E_- \}$.  These are eigenvectors of the Coxeter element $\Ad P_0$ 
with eigenvalues $\om,\om^{-1}$,  as
$P_0=\diag( \om^{n/2},\om^{(n-2)/2},\dots,\om^{-n/2})$.

The \ll $i$-th Coxeter Plane\rrr, which will play a role in the next section, can be described
in the same way as $\Span_\R\{ E^i_+ + E^i_-, \ii(E^i_+ - E^i_-) \}$.

We can go further and compute the root space decomposition
$\g=\h^\pr\oplus (\oplus_{\be\in\De^\pr} \g_\be)$.  By definition,
$\Om E_{i,j} \Om^{-1}$ is a root vector corresponding to $ij$. The Stokes factors $Q_\th$ of Theorem \ref{guitli} are expressed in terms of such root vectors. As matrices they are therefore rather unwieldy, 
but we note that $\Om E_{i,j} \Om^{-1}$ has the following simpler, but rather curious, expression:

\begin{proposition}
$\Om E_{i,j} \Om^{-1} = \tfrac1{n+1} \, d_{n+1}^i \ (\sum_{0\le i,j\le n} E_{i,j})\  d_{n+1}^{-j}$.
\end{proposition}

\begin{proof}
Let us denote the matrix $\sum_{0\le i,j\le n} E_{i,j}$ by $\{1\}$ (all entries of this matrix are $1$).
It is easy to verify that
\[
(n+1) \Om
\bsp
1 & & & \\
 & \ 0\  & & \\
  & & \ddots & \\
   & & & \ 0
\esp
=\{1\}\Om
\]
and that
\[
E_{i,j} = E_-^i
\bsp
1 & & & \\
 & \ 0\  & & \\
  & & \ddots & \\
   & & & \ 0
\esp
E_+^j,
\quad
\textstyle
E_-=E_+^T=  E_{0,n} + \sum_{i=1}^n E_{i,i-1}.
\]
Hence $(n+1)E_{i,j} = 
E_-^i \Om^{-1} \{1\} \Om  E_+^j = \Om^{-1} (\Om E_-^i \Om^{-1}) \{1\} (\Om  E_+^j \Om^{-1}) \Om$.
This is 
$\Om^{-1} d_{n+1}^i \{1\} d_{n+1}^{-j} \Om$
as $\Om E_- \Om^{-1} = d_{n+1}$, $\Om  E_+ \Om^{-1}= d_{n+1}^{-1}$.
We obtain the stated formula.
\end{proof}

\section{Particles and polytopes}\label{3}

The Coxeter Plane first appeared in Toda field theory for a complex semisimple Lie algebra $\g$, in the work of Freeman \cite{Fr91},
Braden et al.\  \cite{BCDS90} and Dorey \cite{Do91},\cite{Do92}.  As we have indicated in the introduction, \ll particles\rr were proposed to correspond to orbits of roots under the action of the Coxeter element $\ga$. Thus, in this particular theory, there are $l$ ($=\rank G$) particles, one for each orbit. The relevant polytope is the polytope in $\h^\ast$ spanned by the roots $\al\in\De$ (equivalently, the polytope in $\h$ spanned by the $H_\al$). 

There are conserved quantities, called $n_i$-spin, where the numbers
$1= n_1\le \cdots \le n_l=s-1$ are the exponents of $\g$.  These arise because it is
known that the eigenvalues of $\ga$ on $\h$ are $e^{2\pi\i n_1/s},\dots,e^{2\pi\i n_l/s}$.
The $n_i$-spin of the orbit of $\be$
is the length of its projection on the \ll $i$-th Coxeter Plane\rrr, where the latter is
defined to be the real $2$-plane whose complexification is spanned by the eigenvectors for
$e^{\pm 2\pi\i n_i/s}$.  (We have $n_i + n_{l+1-i}=s$.)

The case $i=1$ is (another description of) the usual Coxeter Plane, as the 
eigenvectors in question are just $E_\pm$.  
The conserved quantity in this case is called mass.  

As we are mainly interested in the usual Coxeter Plane and the masses of the particles, we
shall take  
\begin{align*}
\text{particle} \  &\longleftrightarrow\   [\be] \ (=\text{ orbit of $\be(E_+)$})
\\
\text{mass of particle} \  &\longleftrightarrow\   \vert \be(E_+) \vert \ (=\text{ distance of $\be(E_+)$ from origin})
\end{align*}
as the basic correspondence. This is an oversimplification, because two distinct root orbits may coincide after
projection to the Coxeter Plane. Furthermore,  it is possible to have two such projections with
the same mass.  These \ll degeneracies\rr may be undesirable from the physical viewpoint. However,
they cause no difficulties for the differential equation interpretation.
\begin{figure}[h]
\begin{center}
\vspace{-2cm}
\includegraphics[scale=0.4, trim= 0 200 0 50]{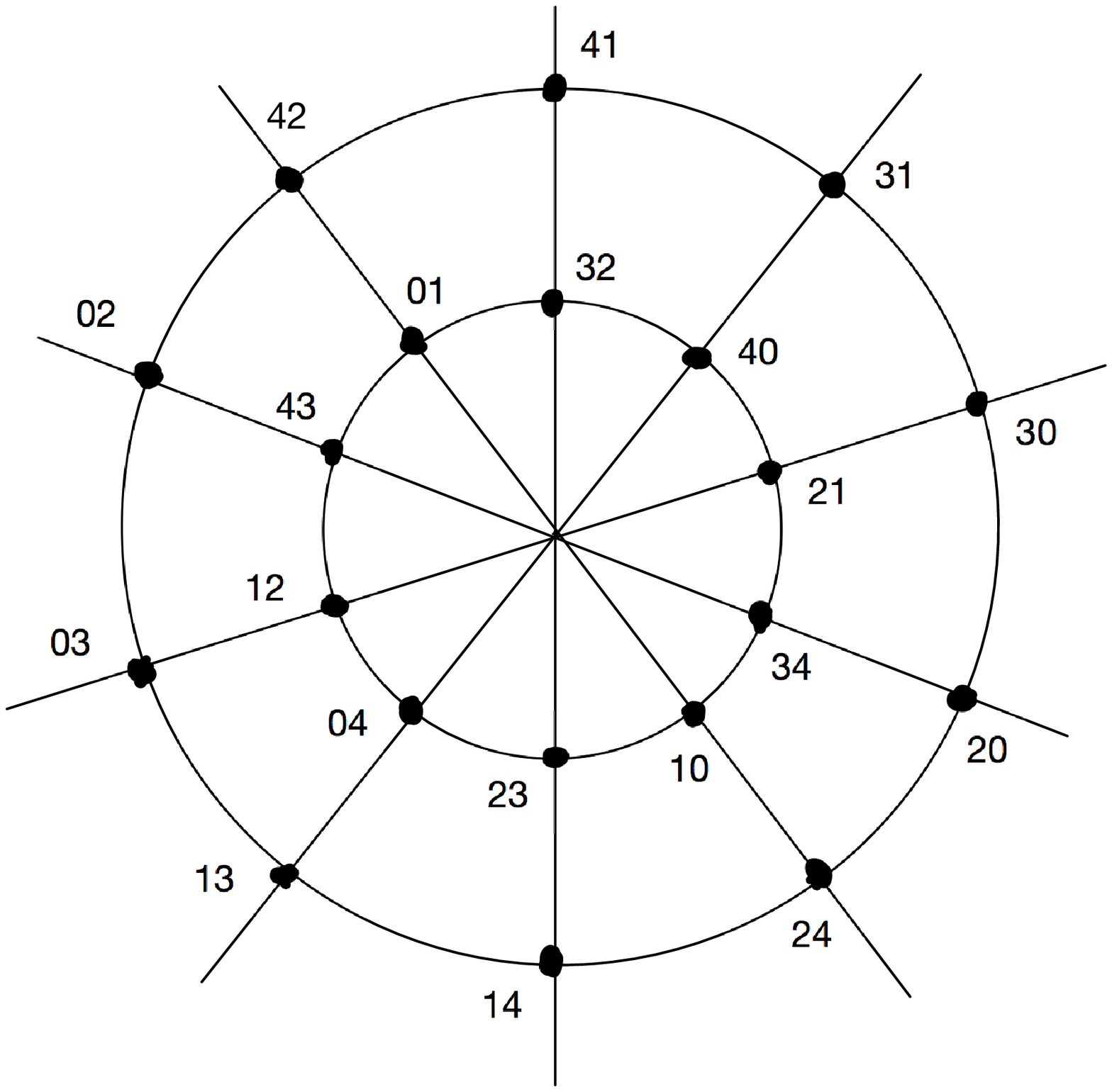}
\end{center}
\caption{Coxeter Plane for $\g=\sl_5\C$.}\label{cox5}
\end{figure}

\begin{example}\label{sl5c} We continue to use the notation for $\sl_{n+1}\C$ which was
introduced in Example \ref{sl4c}. 
The Coxeter Plane for $\sl_5\C$ is shown in Figure \ref{cox5}. 
The notation $ij$ is explained in section \ref{2.4}.  We emphasize that
$ij$ in our Coxeter Plane diagrams indicates the point $ij(-E_+)$.
The action of the Coxeter element on roots $ij$ is given by the permutation $(43210)$, which corresponds to rotation through $-2\pi/5$. 

By section \ref{2.4}, $ij(-E_+)$ is
$(x_i-x_j)(-\diag(1,\om,\om^2,\om^3,\om^4))$, where $\om=e^{2\pi\i/5}$. 
The $20$ roots are distributed on $4$ wheels, with $5$ on each wheel. The
radii of the wheels are
$2\sin\frac\pi5,2\sin\frac{2\pi}5$. Thus, although the roots are represented faithfully 
by the points in the Coxeter Plane, there
is a \ll mass degeneracy\rrr: the particle $[\be]$ and its \ll anti-particle\rr $[-\be]$ lie
on different wheels but with the same radius.
In contrast, in Example \ref{sl4c}, we had $[\be]=[-\be]$, but the \ll particle degeneracy\rr (arising on projecting to the Coxeter Plane) resulted again in only $2$ distinct masses.
\qed
\end{example}

A variant of this physical proposal (apparently with similar motivation) was studied by
Fendley, Lerche, Mathur, and Warner 
(\cite{FLMW91},\cite{LeWa91}).  In their \ll polytopic models\rrr, a finite-dimensional
representation $\th$ of the Lie algebra $\g$ on a vector space $V$ is chosen, and the 
relevant polytope is the polytope in $\h^\ast$ spanned by the weights of the representation.
The weight vectors (in $V$) are taken to be the vacua of the theory. Particles arise
as \ll solitons\rr tunnelling between vacua:  a soliton connects two vacua $v_i,v_j$ if
and only if the corresponding weights $\la_i,\la_j$
differ by a single root, i.e.\ $\la_i-\la_j\in\De$.
The physical characteristics of this particle are those assigned to it in the preceding discussion.

\begin{example}\label{sl2c}
For $\g=\sl_2\C$ the \ll Coxeter Plane\rr is just $\h_\sharp\cong\R$ itself,
with the points $\pm(x_0-x_1)(-\diag(1,-1))=\pm 2$ marked.  This represents one particle,
with mass $2$.  The weights of the irreducible representation of dimension $k+1$
are $kx_0, (k-1)x_0+x_1,\dots,kx_1$ and their projections to the Coxeter Plane are $-k,-k+2,\dots,k$.  Two weight vectors are connected by a soliton --- which can only be the aforementioned particle --- if and only if their (projected) weights differ by $\pm 2$, i.e.\ are adjacent.  
\qed
\end{example}

In view of the o.d.e.\ interpretation of the Coxeter Plane (which did not arise in the above articles
\cite{Fr91},\cite{BCDS90},\cite{Do91},\cite{Do92}),  it becomes plausible that the 
Stokes data (and the solutions themselves) may have some physical meaning.  This indeed turns out to be the case, as we shall see in the next section.

\section{Supersymmetry and integrable systems}\label{4}

\subsection{The tt* equations}\label{4.1}\ 

We turn now to the physics of supersymmetric field theories, at least, the small part of the theory referred to in the introduction. We have observed that the o.d.e.\ of section \ref{2}
fits well with the Lie-theoretic particle interpretation of section \ref{3}.  We emphasize that this was merely an observation; we shall now suggest an explanation for it.

The meromorphic connection form $\hat\om$ of section \ref{2}
does, not, by itself, shed any light on the physical proposals of  section \ref{3}, despite its promising mathematical structure.  To explain
those proposals another meromorphic connection form $\hat\al$ is needed, and here
$\hat\om$ will play an intermediary role.

The connection form $\hat\al$ appears in connection with the tt* equations 
(topological-antitopological fusion equations) of Cecotti and Vafa \cite{CeVa91}.
It is a smooth $t$-family of connection forms on $\C^\ast$, where $t$ is a complex variable, which
represents a  deformation of a conformal field theory. Like $\hat\om$, $\hat\al$ is a meromorphic connection form with poles at $\la=0$ and $\la=\infty$, but for $\hat\al$ both poles have order $2$.
This is the simplest version; more generally 
there are several complex variables $t_1,\dots,t_k$ in an appropriate moduli space of theories.

The space of vacua of the theory is a finite-dimensional complex vector space with a \ll topological metric\rr (independent of $t$) and an action of \ll chiral operators\rr (which depend holomorphically on $t$). 
Based on general principles, Cecotti and Vafa proposed the tt* equations as the equations for a massive deformation of the (massless) conformal field theory. 

The mathematical object corresponding to a solution  of these equations is a (smooth) family of Hermitian metrics which depend on $\vert t\vert$; it is a
Hermitian metric on the trivial vector bundle over $\C^\ast$ with fibre $V$. 

A brief introduction to these equations can be found in section 4.2 of the companion article \cite{Gu21}.  
For the present discussion, the important point is that these equations are
the condition for the connection form $\hat\al$ to be isomonodromic, that is, its monodromy data is independent of $t$ (even though $\hat\al$ itself, whose coefficients incorporate the Hermitian metric, depends on $t$ through those coefficients).  This isomonodromic formulation 
of the tt* equations was made by Dubrovin --- see equation (2.18) in \cite{Du93}.

\subsection{The tt*-Toda equations}\label{4.2}\ 

The comments above apply to the tt* equations in general; in the case
we want to consider (related to $\hat\om$ and the Coxeter Plane) the
connection form $\hat\al$ is
\begin{equation}\label{alphahat}
\hat\al=
\left[
-\tfrac t{\lambda^2} \Ad(e^w) E_-
- \tfrac1\la 
(tw_t +\bar t w_{\bar t})
+ \bar t \Ad(e^{-w}) E_+
\right]
d\la
\end{equation}
where the (unknown) function $w:\C^\ast\to\h_\sharp$
represents the tt* metric.  

Evidently $\hat\al$ has poles of order two at $\la=0$ and at $\la=\infty$. It corresponds to a linear o.d.e.\ in $\la$ with such singularities.  The monodromy is similar, but more complicated, than for the o.d.e.\ associated to $\hat\om$, because now we will have
Stokes matrices at both $\la=0$ and $\la=\infty$.
The condition for $\hat\al$ to be isomonodromic is nothing but the Toda equation 
\begin{equation}\label{tt*-Toda}
2w_{t \bar t} = - \sum_{i=0}^l \, q_i \, e^{-2\al_i(w)} H_{\al_i}
\end{equation}
for $w$. More precisely, it is the 
two-dimensional affine Toda equations for the Lie algebra $\g$, but with extra symmetries imposed for physical reasons (see \cite{GH2},\cite{Gu21}).   One of these
is the requirement that $w$ is \ll radial\rrr, i.e.\ $w=w(\vert t\vert)$. This makes the p.d.e.\ an o.d.e.,
but a rather nontrivial one --- it is a system of nonlinear equations of
Painlev\'e type.
We refer to (\ref{tt*-Toda}) with these extra symmetries as the tt*-Toda equations.

The (physical) relation between $\hat\al$ and $\hat\om$ is that 
$\hat\om$ represents the conformal field theory
being deformed.  The coefficient matrix $\eta$ is the matrix of multiplication by a generator of the
chiral ring of that theory. Roughly speaking, $\hat\al$ is obtained by combining two
copies of $\hat\om$, more precisely $\hat\om$ and its complex conjugate. As
$\hat\om$ is \ll topological\rrr, and its complex conjugate is 
\ll anti-topological\rrr, this gives rise to the terminology tt* fusion.  The Hermitian
metric is involved in the fusion process. 

Mathematically the two connection
forms are related by a factorization procedure - the Iwasawa factorization (but in the loop group of $G$, 
i.e.\ an affine Kac-Moody group, rather than in $G$ itself).  
This is described in detail in \cite{GIL3},\cite{GH2}. 
As a consequence, 
$\hat\om$ is an approximation to $\hat\al$ at $\la=0$ -- in particular {\em they have the same Stokes data there.}
The parameter $t$ of $\hat\al$ is
related to the parameter $z$ of $\hat\om$ by 
\begin{equation}\label{tandz}
t=\tfrac sN c^{1/s} z^{N/s}
\end{equation}
where $c$ is an expression in the
coefficients $c_0,\dots,c_l$.  

\subsection{Solutions of the tt*-Toda equations}\label{4.3}\ 

Now it is time to discuss {\em solutions} of the tt* equations.  This is a nontrivial matter as
solutions of Painlev\'e equations tend to have many singularities, but we will need solutions
which are smooth for $0<\vert t\vert<\infty$. 

Physically, a solution is a massive deformation of a conformal field theory, and the very existence of such a deformation says something about that theory.
Cecotti and Vafa made a series of conjectures about the solutions and their expected physical properties
 (see \cite{CeVa91}, section 8, \ll The magic of the solutions\rrr). In particular:

\no(I) there should exist (globally smooth) solutions $w=w(\vert t\vert)$ on $\C^\ast$,
i.e. smooth for $0<\vert t\vert < \infty$;

\no(II) these solutions should be characterized by asymptotic data at $t=0$
(the \ll ultra-violet point\rrr; here the data is the chiral charges, essentially the $k_i$ in our situation);

\no(III)  these solutions should equally be characterized by asymptotic data at $t=\infty$,
(the \ll infra-red point\rrr; here the data is the soliton multiplicities,  the $s_i$ in our situation).

From now on we focus on the case $\g=\sl_{n+1}\C$, where results are available from
\cite{GH1},\cite{GH2},\cite{GIL1},\cite{GIL2},\cite{GIL3},\cite{GuLi14},\cite{MoXX},\cite{Mo14}.
We expect that similar results will hold for any $\g$. 

Using the notation of Example \ref{sl4c}, we can write $w=\diag(w_0,\dots,w_n)$ where
$w_0,\dots,w_n$ are real functions of $\vert t\vert$.  
The results can be summarized as follows (see section 3 of \cite{Gu21} for a more detailed summary):

\no(I) the global solutions $w$ are in one-to-one correspondence with points $(k_0,\dots,k_n)$ with
all $k_i\ge-1$ (where $N=n+1+ \sum_{i=0}^n k_i$ is now fixed, with $N>0$, e.g.\ $N=1$).  

\no(II) the solution corresponding to $(k_0,\dots,k_n)$
is characterized by its asymptotic behaviour $w_i\sim -m_i \log x$ 
as $t\to 0$ (recall that $m_{i-1}-m_i= \frac{n+1}N(k_i+1) - 1)$).

\no(III) 
the solution corresponding to $(k_0,\dots,k_n)$
is characterized by the following asymptotic behaviour of 
$w_0,\dots,w_n$
as $t\to \infty$:
\begin{equation*}
-\tfrac 4{n+1}
\sum_{p=0}^{[\frac12(n-1)]} w_p \sin \tfrac{(2p+1)k\pi}{n+1}  
\sim
 s_k\, F(L_k x),
\quad
k=1,\dots,[\tfrac12(n+1)]
\end{equation*}
where
$
F(x)=\tfrac12(\pi x)^{-\frac12}e^{-2x}$ and 
$L_k=2\sin \tfrac k{n+1}\pi$. 
(This determines the asymptotics of each $w_i$.)
Here, $[\tfrac12(n+1)]$ means $\tfrac12(n+1)$ if $n$ is odd, and
$\tfrac12n$ if $n$ is even. 
The real numbers $s_1,\dots,s_n$ (with $s_i=s_{n+1-i}$) are the
Stokes numbers of Definition \ref{groups}.  Explicitly,  
\begin{equation}\label{esf}
s_i=\si_i\left(
e^{ (2m_0+n)\frac{\pi\i}{n+1}  },
e^{ (2m_1+n-2)\frac{\pi\i}{n+1}  },
\dots,
e^{ (2m_n-n)\frac{\pi\i}{n+1}  }
\right)
\end{equation}
where $\si_i$ denotes the $i$-th elementary symmetric function of $n+1$ variables.

Thus the global solutions implement the one-to-one correspondence (predicted by Cecotti and Vafa) between

($0$) the data $k_0,\dots,k_n$ (with $\sum_{i=0}^n k_i$ fixed) at $t=0$, and

($\infty$) the data $s_1,\dots,s_n$ at $t=\infty$.

\no We have already noted that ($0$) represents the chiral ring in the UV limit; this
is the information contained in the connection form $\hat\om$.  (The $c_i$ in $\hat\om$ are
uniquely determined by the $k_i$, when we have global solutions.)  

What is the meaning of the data ($\infty$)? The key to this is the expression on the left hand side of the asymptotic formula (III), which arises as follows.  

First, although it is conventional to write
$w=\diag(w_0,\dots,w_n)$ in the usual matrix representation of equation (\ref{tt*-Toda}), $w$ can be expressed more intrinsically as a linear combination of the elements
\[
\textstyle
\{
\sum_{ij\in[\be]} \ e_{ij} \ \in \ \h_\sharp
\st \be\in\De^\pr
\}
\]
where $ e_{ij}$ is a (suitably normalized) root vector for the root $ij$
(cf.\ section 4.3 of \cite{Gu21}). These elements correspond to the particles $[\be]$. 
The coefficient of the $k$-th element in this \ll Fourier expansion\rr of $w$ is then
the left hand side of (III).

Thus the $s_k,L_k$ on the right hand side of (III) are naturally associated to the $k$-th particle.  
As we have seen, $L_k$ is the mass of the $k$-th particle; $s_k$ is called the soliton multiplicity. 
It is a \ll solitonic attribute\rr of the particle, which is consistent with the soliton theory mentioned in section \ref{3} (we shall give
examples in the next section).  Although solitons appeared in that theory, they were not linked to solutions of differential equations there.  On the other hand, in the work
of Cecotti and Vafa on the tt* equations, the Coxeter Plane did not play any role.  Our results on
the tt*-Toda equations are therefore providing supporting evidence for the role of solitons
 in both contexts.

It should be noted that the mass $L_k$ depends only on the particle, and is purely Lie-theoretic. 
On the other hand, the 
soliton multiplicity $s_k$ depends both on the particle and on the particular model given by solution of the tt*-Toda equations.

\section{Examples}\label{5}

We present some examples illustrating the particle/soliton structure
and their differential equation origins, in the case $\g=\sl_{n+1}\C$,
i.e.\ for the Lie algebra of type $A_n$.

\subsection{Particles and their masses}\label{5.1}\ 

As explained in section \ref{2.4}, the points of the Coxeter Plane 
can be computed as the complex numbers
\[
(x_i-x_j) (-\diag(1,\om,\dots,\om^n)) = \om^j-\om^i,
\quad
\om=e^{2\pi\i/(n+1)}.
\]
We have $\vert \om^i - \om^j\vert = 2\sin\frac{\vert i-j\vert}{n+1}\pi$,
and this is the mass of the corresponding particle.
The Coxeter element, represented by the permutation $(n\, n-1\, \cdots\, 2 1 0)$,
acts on the Coxeter Plane by rotation through $-\frac{2\pi}{n+1}$.

If $n+1=2m$, there are $m=\tfrac12(n+1)$ wheels, with radii
\[
2\sin\tfrac{\pi}{n+1},2\sin\tfrac{2\pi}{n+1},\dots,2\sin\tfrac{m\pi}{n+1}\ (=2).
\]
Two Coxeter orbits project to each wheel; this is a \ll particle degeneracy\rr in the
Coxeter Plane.

If $n+1=2m+1$, there are $n+1$ wheels, with radii
\[
2\sin\tfrac{\pi}{n+1},2\sin\tfrac{2\pi}{n+1},\dots,2\sin\tfrac{m\pi}{n+1}\ (<2).
\]
In this case each Coxeter orbit in $\De$ projects to a different wheel in the 
Coxeter Plane, but there is a 
\ll mass degeneracy\rrr:  $[ij]$ and $[ji]$ are different orbits of 
the same radius. 
We regard $([ij],[ji])$ as a \ll particle-antiparticle pair\rrr.
(In contrast, when $n+1=2m$, we have $[ij]=[ji]$.)

In both cases, we have $m$ distinct masses
\[
L_i=2\sin \tfrac {i\pi}{n+1}, \quad 1\le i\le m,
\]
which we may regard as associated either to a particle or
a particle-antiparticle pair.   There are also $m$ soliton multiplicities $s_i$ (Stokes numbers),
which we consider next.

\subsection{Solitons}\label{5.2}\ 

The polytopic models of section \ref{3} involve a choice of representation $\th$.
In Example \ref{sl2c} we considered the rather trivial case $\g=\sl_2\C$ and $\th=S^k\la_2$ (its irreducible representations --- the $k$-th symmetric powers of the standard representation on $\C^2$).

For $\g=\sl_{n+1}\C$, it turns out that the representations $\th=\wedge^k\la_{n+1}$
have special significance. They are the basic irreducible representations, and we have seen
already that they arise in the calculation of the Stokes numbers (Definition \ref{groups}). 

Let us examine the proposal of Fendley et al.\  for these representations.  The weights are $x_{i_1}+\cdots+x_{i_k}$, $0\le i_1 < \dots < i_k \le n$, and we
have corresponding weight vectors $e_{i_1}\wedge\dots\wedge e_{i_k}$
(where $e_0,\dots,e_n$ is the standard basis of $\C^{n+1}$). 

Now we project the weights onto the Coxeter Plane. We obtain the points
\[
(x_{i_1}+\cdots+x_{i_k})  (-\diag(1,\om,\dots,\om^n)).
\]
As in the case of the roots, this produces a diagram with spokes and wheels, which depends on our particular $\th$. 

Fendley et al.\ regard the weight vectors as vacuum vectors, and propose that a soliton tunnels between two weight vectors precisely when the difference of the corresponding weights is a (single) root. Regarding the soliton as a particle, its mass (and higher spins) are those of that root --- i.e.\ the
length of the straight line connecting the relevant points in the diagram. 
The theory of section \ref{4.2} (based on the solutions of the tt*-Toda equations) allows us to assign the soliton multiplicity $s_k$ to any soliton of the $k$-th particle type.

\begin{example}\label{k=1} If $k=1$, 
the weights are $x_0,\dots,x_n$ and the diagram consists
of the $(n+1)$-th roots of unity in the complex plane. 
Any (distinct) vacua $e_i$ and $e_j$ are connected by a soliton; that soliton
has the characteristics of a particle of type $x_i-x_j$, with mass $L_{\vert i-j\vert}$. 

The solitons are illustrated in Figure \ref{gr14}, in the case $n+1=4$.  
The projection of the weight $x_i$ is
$x_i(-\diag(1,\om,\om^2,\om^3))$. 
Thus $x_0,x_1,x_2,x_3$ give the complex numbers $-1,-i,1,i$ respectively; they
are denoted by $0,1,2,3$ in the first part of the figure.
As we have seen in Figure \ref{cox4}, there are two particle types, 
\begin{align*}
[01]&= \{ 01, 30, 23, 12 \} = \{ 10, 03, 32, 21 \} = [03]
\\
[02]&=\{ 02, 31, 20, 13 \}.
\end{align*}
The first part of Figure \ref{gr14}
shows the projections of the weights. The second part
shows (as heavy lines) the four solitons of type $[01]$ (with mass $2\sin\frac \pi 4 = \sqrt 2$).
For example, the points denoted $3,2$ are connected by a soliton of type $[01]=[03]$
because $\pm(x_3-x_2)$ is a root in the set $[01]=[03]$.
The third part shows the two solitons of  type $[02]$ (with mass $2\sin\frac \pi 2 = 2$).
In this example, any two points are connected by a soliton.
\qed
\end{example}
\begin{figure}[h]
\begin{center}
\vspace{-2.5cm}
\includegraphics[angle=90,origin=c,scale=0.4, trim= 500 50 0 50]{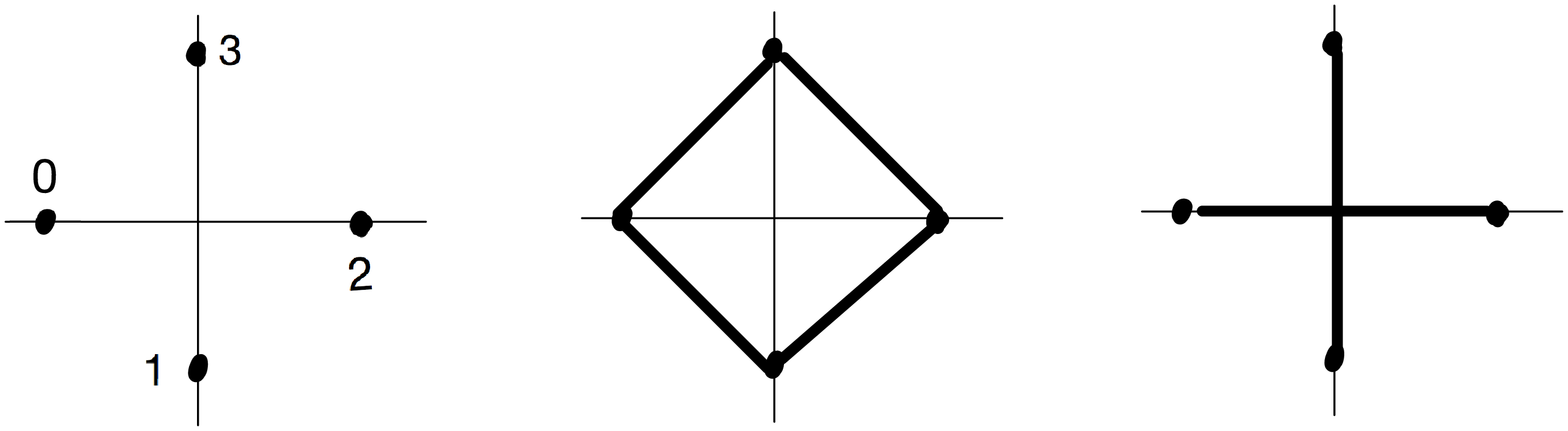}
\end{center}
\caption{Solitons for $\th=\la_{4}
$.}\label{gr14}
\end{figure}

\begin{figure}[h]
\begin{center}
\vspace{5cm}
\includegraphics[angle=270,origin=c,scale=0.4, trim= 500 50 900 50]{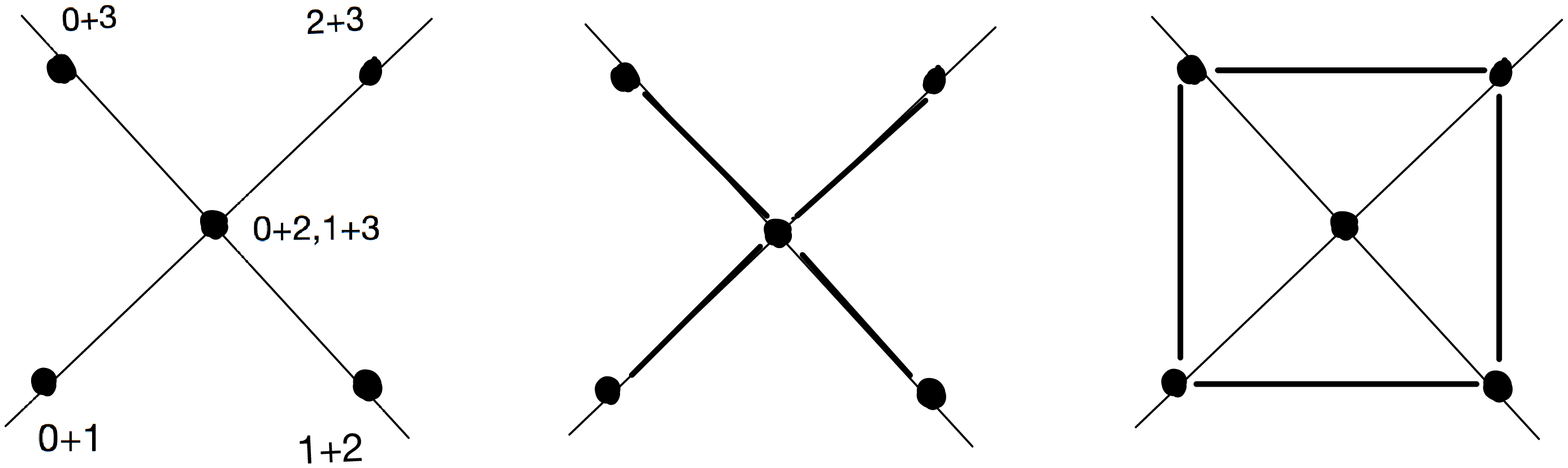}
\end{center}
\caption{Solitons for $\th=\wedge^2\la_{4}
$.}\label{gr24}
\end{figure}

\begin{example}\label{k=2} If $k=2$, 
the weights are $x_i+x_j$ with $0\le i\ne j\le n$.  We obtain the points
$-\om^i -\om^j$ in $\C$. 

The solitons are illustrated in Figure \ref{gr24}, in the case $n+1=4$.  
The first part of the figure
shows the projections of the weights. We denote $(x_i+x_j)(-\diag(1,\om,\om^2,\om^3))$
by $i+j$,  to avoid confusion with the notation $ij$ which means
$(x_i-x_j)(-\diag(1,\om,\om^2,\om^3))$.  
The second part
shows the four solitons of type $[01]$ (with mass $\sqrt 2$). The
third part shows the four solitons of type $[02]$ (with mass $2$). 

The vacua $0+2,1+3$ at the origin are not connected by a soliton, as $x_0+x_2 -(x_1+x_3)$
is not a root.  Neither are the vacua $0+1,2+3$, nor the vacua $0+3,1+2$.
\qed
\end{example}

A more complicated example can be seen in Figure 2 of \cite{Gu21}. All these examples
were described already, in more physical language, in \cite{Bo95a}. Many more are
shown in \cite{LeWa91}.

\subsection{Concrete examples}\label{5.3}\ 

We have been rather evasive about \ll field theories\rr so far, preferring to discuss just
a few ingredients which have mathematical counterparts.  Mathematicians
are advised to regard the particle/mass/soliton framework as a collection of features shared by various 
such theories, rather than a description of a particular theory.  Affine Toda field theory was certainly prominent in the physics papers listed in the references, but many other theories have been
developed since then, and relations between them have been a source of inspiration (and surprises).
With this in mind, we conclude by mentioning two concrete examples, or at least some mathematical aspects of these examples, which illustrate the depth of the framework --- and also its relevance to geometry.  

According to Cecotti and Vafa, the \ll physically realistic\rr solutions of the tt* equations are those with  integer soliton multiplities. In the case of the tt*-Toda equations, where the solutions are as described in section \ref{4.3},  this means that all $s_i\in\Z$.  This is a strong condition, which they proposed and used in \cite{CeVa92} as a way of classifying certain theories.  

In the case of $\g=\sl_{n+1}\C$ there are finitely many solutions (increasing with $n$) which satisfy this condition. We shall give two simple but important examples. A few more examples can be found in \cite{GuLi12}, but the physical/geometrical meaning of most \ll integer Stokes solutions\rr has yet to be investigated.

\no{\em (1) The $A_n$ minimal model}

This is a Landau-Ginzburg theory, one of a family of examples which has an ADE classification
based on Dynkin diagrams of Lie algebras.  These have an important geometrical incarnation: the theory of deformations of singularities (of ADE type) --- see, for example, \cite{Eb07}, \cite{He03}.  

Recall (from section \ref{3}) that our global solutions are parametrized by real matrices
$m=\diag(m_0,\dots,m_n)$ with
\[
m_{i-1}-m_i= \tfrac{n+1}N(k_i+1) - 1 \ge -1
\]
(i.e.\ $k_i\ge-1$).  Let us consider the solution corresponding to
\[
m=-\tfrac1{n+2} \left(\tfrac n2, \tfrac n2-1, \dots, -\tfrac n2\right)
\]
which certainly satisfies the given condition, with
$k_0=1, k_1=\cdots= k_n=0$ and $N=n+2$.

According to our formula (\ref{esf}), the Stokes numbers are
the elementary symmetric functions of the $n+1$ complex numbers
\[
e^{\frac{\pi\i}{n+2} n},
e^{\frac{\pi\i}{n+2} (n-2)},
\dots,
e^{\frac{\pi\i}{n+2} (-n)}.
\]
If $n+1$ is even, these are precisely the $(n+2)$-th roots of $-1$ excluding $-1$ itself. Thus they are the roots of the polynomial
\[
\frac{x^{n+2} + 1}{x+1} = x^{n+1} - x^n + x^{n-1} - \cdots - x + 1.
\]
It follows that all $s_i=1$. 
If $n+1$ is odd, they are the $(n+2)$-th roots of $1$ excluding $-1$, i.e.\  
the roots of the polynomial
\[
\frac{x^{n+2} - 1}{x+1} = x^{n+1} - x^n + x^{n-1} - \cdots + x - 1.
\]
Again we have all $s_i=1$. In view of the one-to-one correspondence between the data 
$(m_0,\dots,m_n)$ at $t=0$ and the data 
$(s_1,\dots,s_n)$ at $t=\infty$, this particular solution is characterized by the natural condition that all soliton multiplicities are $1$.

The data at $t=0$ --- the chiral ring --- has a geometrical interpretation. The space of vacua can be regarded
as the vector space
\[
\C[x]/(f^\pr(x)) \cong \Span\{ 1,x,\dots,x^n\},
\]
the Jacobian ring of the function $f(x)=\frac1{n+2} x^{n+2}$, which
has an isolated singularity of type $A_{n+1}$ at $x=0$. 
It has a deformation (unfolding) given by the family of
functions $f(x,z)= \frac1{n+2} x^{n+2} -zx$, and the Jacobian ring of this family is
\[
\C[x,z]/(x^{n+1}-z).
\]
The matrix of multiplication by $x$ (with respect to the above basis) is
\[
\bp
 & & & z\\
1 & & & \\
  & \ddots & & \\
   & & 1 &
\ep
\]
and this is exactly the matrix $\eta$ from $\hat\om$ (with
$k_0=1, k_1=\cdots= k_n=0$). We ignore the parameters $c_0,\dots,c_n$ here as
they do not affect the isomorphism type of the chiral ring.

\no{\em (2) The Grassmannian sigma model}

This is a nonlinear sigma model, i.e.\ fields are maps from a surface to a
K\"ahler manifold, and the classical equations of motion are the harmonic
map equations, which are the Euler-Lagrange equations for the energy functional.  Holomorphic maps are the harmonic maps with minimal energy (in a given connected component), and
the geometry of moduli spaces of such maps leads in mathematics to the theory of
Gromov-Witten invariants and quantum cohomology of the target space. Calabi-Yau 
manifolds are the most prominent example, but Fano manifolds (such as Grassmannians)
are also important.  

The case of $\C P^n$ first arose in the work of Witten, and Vafa, then Intriligator proposed the
generalization to Grassmannians (see \cite{In91} for this, and further references).  An extensive theory has been developed by algebraic geometers. For this, and some of the relations with
physics,  we refer to \cite{CoKa99}. 

As an example of this situation, we consider  the solution corresponding to
\[
m=-\left(\tfrac n2, \tfrac n2-1, \dots, -\tfrac n2\right).
\]
This corresponds to having
$k_0=0, k_1=\cdots= k_n=-1$. and $N=1$.

By formula (\ref{esf}), the Stokes numbers are just
the elementary symmetric functions of the $n+1$ complex numbers
\[
1,1,\dots,1,
\]
i.e.\ $s_i=\tbinom{n+1}{i}$. 

In the previous example we did not specify a representation $\th$ of $\sl_{n+1}\C$, but the
standard representation $\la_{n+1}$  was implicit. Now we consider the
representation $\th=\wedge^k\la_{n+1}$, in order to relate our solution of the tt*-Toda equations to the \ll polytopic models\rr of section \ref{3}. 

The space of vacua is the vector space $\wedge^k \C^{n+1}$, which can be identified naturally with
the cohomology
\[
H^\ast( Gr_k(\C^{n+1}); \C)
\]
of the Grassmannian (as $Gr_k(\C^{n+1})$ is a \ll minuscule flag manifold\rrr).  The cup
product on this vector space has a deformation given by the quantum product, and the matrix of quantum multiplication by a generator of $H^2( Gr_k(\C^{n+1}); \C)$ is known to be
$\wedge^k z\eta$.  References for these facts are \cite{GoMaXX}, \cite{LTXX}, \cite{KaXX}.

For $k=1$, we have
\[
z\eta=
\bp
 & & & z\\
1 & & & \\
  & \ddots & & \\
   & & 1 &
\ep
\]
which is the well known matrix of quantum multiplication by a generator $x$ of $H^2( \C P^n; \C)$
in the quantum cohomology ring $QH^\ast( \C P^n; \C) \cong \C[x,z]/(x^{n+1}-z)$.  This
is the chiral ring for $k=1$. It is isomorphic to the chiral ring in the previous example.
(Using $z\eta$ instead of $\eta$ is just a matter of interpretation: in the
first example the matrix-valued form $\eta dz$ arises naturally, while $z\eta \frac{dz}z$
arises in the quantum cohomology examples. The relation of this form with $\hat\om$ is explained in section 2 of \cite{GIL3}.)

The fact that the {\em same} solution of the tt*-Toda equations gives rise to (part of) the
quantum cohomology of $Gr_k(\C^{n+1})$ for {\em all} $k$ is related to the
\ll Satake Correspondence\rrr --- see 
\cite{Gu21} for more information and further references.

We remark that most aspects of the Grassmannian example were well known to physicists 30 years ago, as
can be seen from the articles
\cite{Bo95a}, \cite{Bo95b}.  Apart, that is,  from the existence of the corresponding global solutions, which crucially
relate the data at $t=0$ and the data at $t=\infty$, as well as
the rigorous derivation of their asymptotics. Our results fill this gap.

Our results also give a simpler computation of the Stokes data, based on
the Coxeter Plane. In the physics literature, the Stokes data first arose in the context of Landau-Ginzburg models, using techniques of singularity theory (Picard-Lefschetz theory). The Grassmannian example (which is also a Landau-Ginzburg model) provides an example of this. Namely, physicists discovered (\cite{Ge91}, \cite{In91}, \cite{Va92}, \cite{Wi95}) that the quantum cohomology of 
$Gr_k(\C^{n+1})$ has a \ll Landau-Ginzburg presentation\rr
\[
QH^\ast(Gr_k(\C^{n+1});\C) \cong \C[x_1,\dots,x_k]/
(\tfrac{\b W}{\b x_1},\dots,\tfrac{\b W}{\b x_k})
\]
where $W=W(x_1,\dots,x_k)$ is a \ll superpotential\rrr.  Here $x_1,\dots,x_k$ are the Chern classes of the tautologous bundle.   We are omitting explicit mention of the quantum parameter $z$ from the notation.

In terms of the Chern roots $u_1,\dots,u_k$ the superpotential
is just
\[
W(u_1,\dots,u_k)= \tfrac1{n+2}(u_1^{n+2}+\cdots+u_k^{n+2}) - z (u_1+\cdots+u_k).
\]
The critical points of $W$ are the $\tbinom{n+1}{k}$ sets $\{\al_1,\dots,\al_k\}$ where 
$\al_1,\dots,\al_k$ are any $k$ (distinct) $(n+1)$-th roots of $z$ (we assume $z\ne 0$). 
The critical values are the $\tbinom{n+1}{k}$ complex numbers
\[
\tfrac1{n+1} z (\al_1+\cdots+\al_k) - z(\al_1+\cdots+\al_k) = -\tfrac n{n+1} z (\al_1+\cdots+\al_k).
\]
These  critical values  form the points of the physicists' \ll $W$-plane\rrr.  It is (up to scalar multiplication) exactly
the plane of the polytopic model explained in sections \ref{3} and \ref{5.2}. That is, it is the (underlying plane of the) Coxeter Plane with points given by evaluating the (apposition) weights of the representation
$\th=\wedge^k\la_{n+1}$ on $E_+$.

\no{\em Acknowledgements:\  }   This article is based on a talk which was scheduled to be given by the first author at the workshop \ll Representations of Discrete Groups and Geometric Topology on Manifolds\rrr, Josai University, 12-13 March 2020, but postponed due to Covid-19 restrictions.  The authors are grateful to the organisers for the opportunity to submit this version.

{\em
\noindent
Department of Mathematics\newline
Faculty of Science and Engineering\newline
Waseda University\newline
3-4-1 Okubo, Shinjuku, Tokyo 169-8555\newline
JAPAN\newline
\ \newline
Department of Mathematics\newline
National Tsing Hua University\newline
Hsinchu 300\newline
TAIWAN\newline
and\newline
National Center for Theoretical Sciences\newline
Taipei 106\newline
TAIWAN
}

\end{document}